\documentclass[1 [leqno, 11pt]{amsart}
\usepackage{amssymb,  amsmath, amsmath, latexsym, amssymb, amsfonts, amsbsy, amsthm,mathtools, graphicx, CJKutf8, CJKnumb, CJKulem, extarrows, color, dsfont}
\usepackage{mathtools,graphicx,CJKutf8,CJKnumb,CJKulem,color}
\usepackage{cancel}
\usepackage{hyperref}
\usepackage{scalerel,stackengine}
\usepackage{float}
\usepackage{multirow}
\usepackage{makecell}
\usepackage{tikz}
\usetikzlibrary{shapes.geometric, arrows,plotmarks,backgrounds}
\usepackage{pgfplots}
\numberwithin{equation}{section}

\allowdisplaybreaks
\let\e=\varepsilon

\let\la=\lambda

\let\f=\frac

\let\om=\omega

\let\D=\Delta

\let\ep=\epsilon
\let\Om=\Omega
\let\wt=\widetilde
\let\wh=\widehat

\let\na=\nabla
\let\th=\theta
\let\pa=\partial
\let\va=\varphi

\def\no{\noindent}

\def\eqdef{\buildrel\hbox{\footnotesize def}\over =}

\def\bbT{\mathbb{T}}
\def\bbZ{\mathbb{Z}}
\def\bbR{\mathbb{R}}

\newcommand{\beq}{\begin{equation}}
\newcommand{\eeq}{\end{equation}}
\newcommand{\beqno}{\begin{equation*}}
\newcommand{\eeqno}{\end{equation*}}
\newcommand{\ben}{\begin{eqnarray}}
\newcommand{\een}{\end{eqnarray}}
\newcommand{\beno}{\begin{eqnarray*}}
\newcommand{\eeno}{\end{eqnarray*}}

\newtheorem{theorem}{Theorem}[section]

\newtheorem{lemma}[theorem]{Lemma}
\newtheorem{proposition}[theorem]{Proposition}
\newtheorem{corol}[theorem]{Corollary}
\newtheorem{remark}[theorem]{Remark}

\begin{document}
\begin{CJK*}{GBK}{song}

\title[Asymptotic stability for Boussinesq]{Asymptotic stability for two-dimensional Boussinesq systems around the Couette flow in a finite channel}
\author{Nader Masmoudi}
\address{NYUAD Research Institute, New York University Abu Dhabi,Saadiyat Island, Abu Dhabi, P.O. Box 129188, United Arab Emirates\\
Courant Institute of Mathematical Sciences, New York University, 251 Mercer Street New York, NY 10012 USA}
\email{masmoudi@cims.nyu.edu}
\author{Cuili Zhai}
\address{School of Mathematics and Physics, University of Science and Technology Beijing, 100083, Beijing, P. R. China.}
\email{zhaicuili035@126.com, cuilizhai@ustb.edu.cn}
\author{Weiren Zhao}
\address{Department of Mathematics, New York University Abu Dhabi, Saadiyat Island, P.O. Box 129188, Abu Dhabi, United Arab Emirates.}
\email{zjzjzwr@126.com, wz19@nyu.edu}
\begin{abstract}
In this paper, we study the asymptotic stability for the two-dimensional Navier-Stokes Boussinesq system around the Couette flow with small viscosity $\nu$ and small thermal diffusion $\mu$ in a finite channel. In particular, we prove that if the initial velocity and initial temperature $(v_{in},\rho_{in})$ satisfies $\|v_{in}-(y,0)\|_{H_{x,y}^2}\leq \e_0 \min\{\nu,\mu\}^{\f12}$  and $\|\rho_{in}-1\|_{H_x^{1}L_y^2}\leq \e_1 \min\{\nu,\mu\}^{\f{11}{12}}$ for some small $\e_0,\e_1$ independent of $\nu, \mu$, then for the solution of the two-dimensional Navier-Stokes Boussinesq system, the velocity remains within $O(\min\{\nu,\mu\}^{\f12})$ of the Couette flow, and approaches to Couette flow as $t\to\infty$;  the  temperature remains within $O(\min\{\nu,\mu\}^{\f{11}{12}})$ of the constant $1$, and approaches to $1$ as $t\to\infty$.
\end{abstract}
\maketitle

\section{Introduction}
In this paper, we consider the two-dimensional Navier-Stokes Boussinesq system  in a finite channel $\Om=\{(x,y): x\in\bbT, y\in(-1,1)\}$:
 \beq\label{eq: NSB}
 \left\{\begin{array}{l}
\pa_tv+v\cdot\na v-\nu\D v+\na P=-\rho ge_2\\
\pa_t\rho+v\cdot\na\rho-\mu\D \rho=0, \ \na\cdot v=0,\\
v(t,x,\pm1)=(\pm1,0),\  \rho(t,x,\pm1)=c_0,\\
v(0,x,y)=v_{in}(x,y), \ \rho(0,x,y)=\rho_{in}(x,y),
\end{array}\right.
\eeq
where $\nu$ is the viscosity coefficient and $\mu$ is the thermal diffusivity, $v(t,x,y)=(v^1,v^2)$ is the two-dimensional velocity field, $P(t,x,y)$ is the pressure, $\rho$ is the temperature, $g=1$ is the normalized gravitational constant and $e_2=(0,1)$ is the unit vector in the vertical direction. The boundary condition in \eqref{eq: NSB} means that the fluid is moving together with the boundary and the temperature is fixed at the boundary. Let us also normalize $c_0=1$ for simplicity.

The system \eqref{eq: NSB} has a flowing steady state
\beq
v_s=(y,0),\quad \rho_s=1,\quad p_s=y+c.
\eeq

Now we introduce the perturbation: $v=u+(y,0)$, $P=p+p_s$ and $\rho=\theta+\rho_s$, then $(u,p,\theta)$ satisfies
\begin{equation}\label{eq: u sys}
 \left\{\begin{array}{l}
\pa_tu+y\pa_xu+\Big(
  \begin{array}{ccc}
   u^2\\
    0\\
  \end{array}
\Big)+u\cdot\na u-\nu\D u+\na p=-  \Big(\begin{array}{ccc}
   0\\
    \theta\\
  \end{array}\Big),\\
\pa_t\theta+y\pa_x\theta+u\cdot\na\theta-\mu\D \theta=0,\ 
\na\cdot u=0,\\
u(t,x,\pm1)=0,\ \theta(t,x,\pm1)=0,\\
u(0,x,y)=u_{in}(x,y),\ \theta(0,x,y)=\theta_{in}(x,y).
\end{array}\right.
\end{equation}
%Here we consider the non-slip boundary condition on the perturbation $u$ in order to include the boundary effect.

We also introduce the vorticity $\om=\na \times u=\pa_yu^1-\pa_xu^2$, which satisfies
\beq\label{eq: vorticity}
 \left\{\begin{array}{l}
\pa_t\om+y\pa_x\om+u\cdot\na \om-\nu\D\om=-\pa_x\theta,\\
\pa_t\theta+y\pa_x\theta+u\cdot\na\theta-\mu\D\theta=0,\\
u=\na^{\bot}\psi=(\pa_y\psi,-\pa_x\psi),\quad \D\psi=\om.
\end{array}\right.
\eeq
Note that  we can not impose the boundary condition on the vorticity, which is the main difficulty of this paper.

Before stating our main result, let us first recall previous works about the stability of flowing steady states. The linear inviscid two-dimensional Boussinesq system with shear flows has been extensively studied starting from the works of Taylor \cite{Taylor}, Goldstein \cite{Goldstein} and Synge \cite{Synge}. We also refer to the book of Lin \cite{Lin}. %It has been proved that the stability of the solutions of the linearized inviscid two-dimensional Boussinesq system with a shear flow is determined by the competition between the stabilizing force and the vertical shear flow $U(y)$. 
The system \eqref{eq: u sys} is well studied in the infinite channel case $\mathbb{T}\times \mathbb{R}$. We can refer to \cite{Bian, DWZ, Zillinger1, Zillinger2, Zillinger3}. The best stability threshold result when $\nu=\mu$ is 
\begin{align}\label{eq:1.5}
\|\om_{in}\|_{H^s}\leq \ep\nu^{\f12}, \ \|\theta_{in}\|_{H^s}\leq\ep\nu,\ \||D_x|^{\f13}\theta_{in}\|_{H^s}\leq\ep\nu^{\f56}, 
\end{align}
with $s>1$, which was proved by Deng, Wu and Zhang \cite{DWZ}.  The mechanisms leading to stability are the so-called inviscid damping and enhanced dissipation, which are well studied for the Navier-Stokes system around Couette flow which we will introduce later. 
 %In  \cite{DWZ}, the authors investigated the stability of the Couette flow for the two-dimensional Navier-Stokes Boussinesq system with both dissipation and thermal diffusion. They also considered the problem with a weaker stabilization mechanism and studied the partial dissipation case. In \cite{Bian}, the authors considered the stability threshold for two-dimensional shear flow of  the Boussinesq system near Couette when the viscosity is different from the thermal diffusivity.  
Without thermal diffusion, Masmoudi, Said-Houari and Zhao \cite{MHZ} considered the Navier-Stokes Boussinesq system with no heat diffusion in the thermal equation,  and they studied the stability of Couette flow for the initial data perturbation in Gevrey-$\f1s$ for $\f13<s\leq1$ in the domain $\bbT\times\bbR$. For the Euler Boussinesq system $\nu=\mu=0$, the global well-posedness for large data is an open problem. In \cite{YL}, Yang and Lin proved the linear inviscid damping for the linearized two-dimensional Euler Boussinesq system which is generalized in \cite{Bianchini}. The nonlinear inviscid damping for large time is studied by Bedrossian, Bianchini, Coti Zelati and Dolce \cite{Bedrossian}.
%Very recently, Zillinger studied  the asymptotic stability of the two dimensional Boussinesq equations with partial dissipation or without thermal dissipation in his works \cite{Zillinger1,Zillinger2, Zillinger3}.

 In this paper, %we consider  the case that the viscosity and the thermal diffusivity are same, i.e. $\nu=\nu$, and 
 we mainly study the boundary effect due to the non-slip boundary condition on the velocity.
%The  Boussinesq system \eqref{eq: NSB} attracted the attention of many mathematicians and physicists. : first, due to its wide range of applications, see for example \cite{} and second, due to the fact that the two-dimensional Boussinesq model retain some key features of the 3D Euler and Navier-Stokes equations. For instance, it has been known that the inviscid two-dimensional Boussinesq equations are identical to the incompressible axi-symmetric swirling 3D Euler equations, as point out in \cite{}. During the last decades, many interesting results were obtained in different directions. 

%By taking the Fourier transform in $x\in\bbT$, we denote that \begin{align*}\theta(t,x,y)&=\sum_{k\in\mathbb{Z}}\wh{\theta}_k(t,y)e^{ikx},\\om(t,x,y)=\sum_{k\in\mathbb{Z}}\wh{w}_k(t,y)e^{ikx},\ u(t,x,y)=\sum_{k\in\mathbb{Z}}\wh{u}_k(t,y)e^{ikx}.%=\sum_{k\in\mathbb{Z}}e^{ikx}(\pa_y^2-k^2)\wh{\psi}_k(t,y),\end{align*}We also define \begin{align*}P_0f=\ol{f}(t,y)=\int_{\bbT}f(t,x,y)dx,\ f_{\neq}=f-\ol{f}=\sum_{k\neq0}\wh{f}_k(t,y)e^{ikx}.\end{align*}

Our main result is stated as follows.
\begin{theorem}\label{main thm}
Suppose that $(u,\theta)$ solves the system \eqref{eq: u sys} with the initial data $(u_{in}, \theta_{in})$. Then there exist constants $\nu_0$ and $\e_0,\e_1, C>0$ independent of  $\nu,\mu$ so that if 
\begin{align*}
\|u_{in}\|_{H^2}\leq \e_0\min\{\nu,\mu\}^{\f12},
\end{align*}
\begin{align*}
 \|\theta_{in}\|_{H^1}+ \||D_x|^{\f16}\theta_{in}\|_{H^1}\leq \e_1\min\{\nu,\mu\}^{\f{11}{12}}, 
\end{align*}
for some sufficiently small $\e_0,\e_1$, $0<\min\{\nu,\mu\}\leq\nu_0$, then the solution $(u,\theta)$  is global in time and satisfies the following stability estimates:
\begin{align*}
&\|(1-|y|)^{\f12}\om\|_{\wt{L}^{\infty}_t\mathcal{F}L^1L^2_y}+\|\pa_xu\|_{\wt{L}_t^2\mathcal{F}L^1L_y^2}+\||D_x|^{\f12}u\|_{\wt{L}_t^{\infty}\mathcal{F}L^1L_y^{\infty}}+\nu^{\f14}\||D_x|^{\f12}\om\|_{\wt{L}_t^2\mathcal{F}L^1L_y^2}\\
&\quad \leq C\e_0\min\{\nu,\mu\}^{\f12},
\end{align*}
and
\begin{align*}
\|\theta\|_{\wt{L}_t^{\infty}\mathcal{F}L^1L_y^2}%+\nu^{\f16}\||D_x|^{\f13}\theta\|_{\mathcal{F}L^1L_t^2L_y^2}
+\||D_x|^{\f16}\theta\|_{\wt{L}_t^{\infty}\mathcal{F}L^1L_y^2}+\mu^{\f16}\||D_x|^{\f23}\theta\|_{\wt{L}_t^2\mathcal{F}L^1L_y^2}\leq C\e_1\min\{\nu,\mu\}^{\f{11}{12}},
\end{align*}
where $\|f\|_{\wt{L}_t^p\mathcal{F}L^1L_y^q}=\sum\limits_{k\in \mathbb{Z}}\|\wh{f}_k\|_{L_t^pL_y^q}$ and $\wh{f}_k=\f{1}{2\pi}\int_{\mathbb{T}}f(x)e^{-ikx}dx$ is the Fourier transform of $f$ in the $x$ direction and $k$ is the wave number. 
\end{theorem}
\begin{remark}
The function space $\wt{L}_t^p\mathcal{F}L^1L_y^q$ is of the same spirit as the Chemin-Lerner's Besov space  \cite{CL}. 
%Here we use the function space $\wt{L}_t^p\mathcal{F}L^1L_y^q$, which was first introduced by Chemin and Lerner \cite{CL}.
\end{remark}
\begin{remark}
The asymptotic stability holds for the initial perturbation satisfying
\begin{align*}
\sum_{k\in\bbZ}\|\wh{w}_{in,k}\|_{L^2}+\sum_{k\in\bbZ\setminus\{0\}}|k|^{-1}\|\pa_y\wh{w}_{in,k}\|_{L^2}\leq C\e_0\min\{\nu,\mu\}^{\f12},
\end{align*}
and 
\begin{align*}
\|\wh{\theta}_{in,0}\|_{L^2}+\sum_{k\in\bbZ\setminus\{0\}}\||k|^{\f16}\wh{\theta}_{in,k}\|_{L^2}\leq C\e_1\min\{\nu,\mu\}^{\f{11}{12}}.
\end{align*}
\end{remark}
\begin{remark}
The estimate $\|\pa_xu\|_{\wt{L}_t^2\mathcal{F}L^1L_y^2}$ is due to the inviscid damping and the estimates $\nu^{\f14}\||D_x|^{\f12}\om\|_{\wt{L}_t^2\mathcal{F}L^1L_y^2}$ and $\mu^{\f16}\||D_x|^{\f23}\theta\|_{\wt{L}_t^2\mathcal{F}L^1L_y^2}$ are due to the enhanced dissipation.
\end{remark}

\begin{remark}
Compared to \cite{DWZ}, when $\nu=\mu$, the interpolation of Sobolev spaces gives that the stability threshold is actually more restrictive than the one in our paper. In \cite{DWZ}, an extra smallness on lower frequencies is required, namely $\|\th_{in}\|_{H^s}\leq\ep \nu$. The key point of improvement is that we are able to control the buoyancy term and nonlocal terms in the temperature equation by avoiding discussing the different sizes of $\theta$ in different frequencies.
\end{remark}
\begin{remark}
If $\th_{in}=0$, $\nu=\mu$, Theorem \ref{main thm} is consistent with the Navier-Stokes result in \cite{CLWZ2020}. We also remark that the stability problem of two-dimensional Couette flow has previously been investigated. One may refer to  \cite{Bedrossian2016, Bedrossian2018, MZ2019, MZ CPDE} for infinite channel case, and to \cite{Bedrossian2019, CLWZ2020} for finite channel case.  In this paper, the linear estimates of the velocity and the vorticity can be obtained by the same method as \cite{CLWZ2020}, and in order to shorten this paper, we will use some linear estimates from \cite{CLWZ2020} as a black box. 
\end{remark}

\begin{remark}
For the Navier-Stokes result, the restriction on the size of perturbations for the asymptotic stability is $\nu^{\f12}$ which was obtained in \cite{CLWZ2020} due to the boundary effect. Without boundary, it is expected that the stability threshold is $\nu^{\f13}$ for perturbations in some higher regularity Sobolev spaces \cite{MZ2019}.  By modifying the time-dependent multiplier of \cite{MZ2019} and treating the bouyancy term as in this paper, one can obtain that for the system \eqref{eq: vorticity} in $\bbT\times\bbR$, the asymptotic stability  holds for larger initial perturbations, namely,
\begin{align*}
\|\om_{in}\|_{H^s}\leq \e_0\min\{\nu,\mu\}^{\f13},
\end{align*}
\begin{align*}
 \|\theta_{in}\|_{H^s}+ \||D_x|^{\f16}\theta_{in}\|_{H^s}\leq \e_1\min\{\nu,\mu\}^{\f{2}{3}},
\end{align*}
 with some $s$ large.
\end{remark}

%In this paper,% for simplicity, we will only present the results of the resolvent estimates of the vorticity in Section \ref{sec: space-time esti}, and ones can refer to \cite{CLWZ2020} for the details of their proof. And 
In order to control the buoyancy term, in section \ref{sec: space-time esti}, we obtain the precise estimates of $\theta$ by decomposing the system of $\theta$ into inhomogeneous problem and homogeneous problem.  For the homogeneous part, we can obtain the sharp bound by using the Gearhart-Pr\"uss lemma in \cite{Wei}. And for the inhomogeneous part, we obtain Proposition \ref{prop: theta I} by some resolvent estimates which were obtained in {\it section 3 of \cite{CLWZ2020}} with the Navier-slip boundary condition. Finally, in section \ref{sec: nonlinear}, we will mainly give the proof of the nonlinear stability.
 % with the zero boundary condition on $\theta$. However, the buoyancy force term $\pa_x\theta$ in the vorticity equation brings the trouble.  Due to the buoyancy term $\pa_x\theta$ in the equation of the vorticity $\om$, in the process of estimating $\|\om(t)\|_{L^2}$ and since the enhanced dissipation in estimate of $\|\theta(t)\|_{L^2}$ contains only $\f13-$horizontal derivative dissipation, we need to estimate $\||D_x|^{\f13}\theta(t)\|_{L^2}$ in order to control the buoyancy term. This explains why we combine the estimates of $\|\om(t)\|_{L^2},\|\theta(t)\|_{L^2}$ and $\||D_x|^{\f13}\theta(t)\|_{L^2}$. On the other hand, in order to deal with the nonlinear term of $\theta$, we  also need to decompose the low and high frequency of $\theta$. And from the procedure of the proof of Theorem \ref{main thm}, we can also find that the index $\f56$ of $\||D_x|^{\f13}\theta_0\|_{L^2}\leq\e_1v^{\f56}$ is optimal in some degree. 

\section{Space-time estimates of the linearized Boussinesq equations}\label{sec: space-time esti}
In this section, we establish the space-time estimates of the linearized two-dimensional Boussinesq equation.
By taking the Fourier transform in $x\in\bbT$, we have
\begin{align*}
\theta(t,x,y)&=\sum_{k\in\mathbb{Z}}\wh{\theta}_k(t,y)e^{ikx},\
\om(t,x,y)=\sum_{k\in\mathbb{Z}}\wh{w}_k(t,y)e^{ikx},\ u(t,x,y)=\sum_{k\in\mathbb{Z}}\wh{u}_k(t,y)e^{ikx}.%=\sum_{k\in\mathbb{Z}}e^{ikx}(\pa_y^2-k^2)\wh{\psi}_k(t,y),
\end{align*}
%In the following, we denote $\theta_k(t,y)=\wh{\Theta}_k(t,y)$, $w_k(t,y)=\wh{W}_k(t,y)$ and  $u_k(t,y)=\wh{U}_k(t,y)$. 
And for convenience,  we suppress the index $k$ in $\wh{\theta}_k, \wh{w}_k, \wh{u}_k$.

\subsection{Space-time estimates for the vorticity}
Let us first study the following system for $k\neq 0$:
\begin{align}\label{eq: om general}
\left\{\begin{aligned}
&\pa_t\wh{w}+\nu(\pa_y^2-k^2)\wh{w}+iky\wh{w}=-ikf^1-\pa_yf^2,\quad w|_{t=0}=\wh{w}_{in}(k,y),\\
&\wh{w}=\pa_y\wh{u}^1-ik\wh{u}^2,\quad \wh{u}(t,k,\pm1)=0.
\end{aligned}\right.
\end{align}

We also introduce the space-time norm:
\begin{align*}
\|f\|_{L^pL^q}=\left\|\|f(t)\|_{L^q(-1,1)}\right\|_{L^p(\bbR^+)}.
\end{align*}

Let us introduce the following estimate for \eqref{eq: om general}.
%For \eqref{eq: om general}, following the same procedure of Proposition 6.1 in \cite{CLWZ2020}, we directly have the following space-time estimates. And for simplicity, we omit the details of its proof.
\begin{proposition}\label{prop: w} (Proposition 6.1 in \cite{CLWZ2020}.)
Let $0<\nu\leq\nu_0$ and $\wh{w}$ be a solution of $\eqref{eq: om general}$ with $\wh{w}_{in}\in H^1(-1,1)$ and $f^1,f^2\in L^2L^2$, where $\wh{w}_{in}$ satisfies $\langle \wh{w}_{in},e^{\pm ky}\rangle=0$. Then there exists a constant $C>0$ independent of $\nu,k$ so that
\begin{align*}
&|k|\|\wh{u}\|^2_{L^{\infty}L^{\infty}}+k^2\|\wh{u}\|^2_{L^2L^2}+(\nu k^2)^{\f12}\|\wh{w}\|^2_{L^2L^2}+\|(1-|y|)^{\f12}\wh{w}\|^2_{L^{\infty}L^2}\\
&\quad\leq C\big(\|\wh{w}_{in}\|^2_{L^2}+k^{-2}\|\pa_y\wh{w}_{in}\|^2_{L^2}\big)+C\big(\nu^{-\f12}|k|\|f^1\|^2_{L^2L^2}+\nu^{-1}\|f^2\|^2_{L^2L^2}\big).
\end{align*}
\end{proposition}

\subsection{Space-time estimates for $\theta$}
First of all, we consider the linearized equation:
\begin{align}\label{eq: theta}
\pa_t\wh{\theta}-\mu(\pa_y^2-k^2)\wh{\theta}+iky\wh{\theta}=-ikg^1-\pa_yg^2, \ \wh{\theta}|_{t=0}=\wh{\theta}_{in},\ \wh{\theta}|_{y=\pm1}=0.
\end{align}

By the standard energy estimates for $\wh{\theta}$, we can easily get the following proposition, which is important for the estimates of high frequency of $\wh{\theta}$.
\begin{proposition}\label{prop: theta energy esti}
Let $\theta$ be a solution of \eqref{eq: theta} with $\wh{\theta}_{in}\in L^2(-1,1)$ and $g^1,g^2\in L^2L^2$. Then there exists a constant $C>0$ independent in $\mu,k$ so that 
\begin{align*}
\|\wh{\theta}\|^2_{L^{\infty}L^2}+\mu k^2\|\wh{\theta}\|^2_{L^2L^2}+\mu\|\pa_y\wh{\theta}\|^2_{L^2L^2}\leq C\mu^{-1}\big(\|g^1\|^2_{L^2L^2}+\|g^2\|^2_{L^2L^2}\big)+\|\wh{\theta}_{in}\|^2_{L^2}.
\end{align*}
\end{proposition}
\begin{proof}
Taking $L^2$ inner product between \eqref{eq: theta} and $\theta$, we get 
\begin{align*}
\langle\pa_t\wh{\theta},\wh{\theta}\rangle-\mu\langle(\pa_y^2-k^2)\wh{\theta},\wh{\theta}\rangle+\langle iky\wh{\theta},\wh{\theta}\rangle=\langle -ikg^1-\pa_yg^2,\wh{\theta}\rangle.
\end{align*}
By taking the real part and integration by parts in the above equality, we obtain
\begin{align*}
\f12\f{d}{dt}\|\wh{\theta}\|_{L^2}^2+\mu\|\pa_y\wh{\theta}\|_{L^2}^2+\mu k^2\|\wh{\theta}\|_{L^2}^2&\leq C\|g^1\|_{L^2}\|k\wh{\theta}\|_{L^2}+C\|g^2\|_{L^2}\|\pa_y\wh{\theta}\|_{L^2}\\
&\leq \f14\mu\|k\wh{\theta}\|_{L^2}^2+C\mu^{-1}\|g^1\|_{L^2}^2+\f14\mu\|\pa_y\wh{\theta}\|_{L^2}+C\mu^{-1}\|g^2\|_{L^2}^2.
\end{align*}
Thus by integrating in time, we have
\begin{align*}
\|\wh{\theta}\|^2_{L^{\infty}L^2}+\mu k^2\|\wh{\theta}\|^2_{L^2L^2}+\mu\|\pa_y\wh{\theta}\|^2_{L^2L^2}\leq C\mu^{-1}\big(\|g^1\|^2_{L^2L^2}+\|g^2\|^2_{L^2L^2}\big)+\|\wh{\theta}_{in}\|^2_{L^2}.
\end{align*}
\end{proof}

In order to deal with the buoyancy term $\pa_x\theta$ in the vorticity equation, we also need to give the following estimates about $\wh{\theta}$.

First, we decompose $\wh{\theta}=\wh{\theta}_I+\wh{\theta}_H$, where $\wh{\theta}_I$ solves
\begin{align}\label{eq: theta I}
\pa_t\wh{\theta}_I-\mu(\pa_y^2-k^2)\wh{\theta}_I+iky\wh{\theta}_I=-ikg^1-\pa_yg^2, \ \wh{\theta}_I|_{t=0}=0,\ \wh{\theta}_I|_{y=\pm1}=0,
\end{align}
and $\wh{\theta}_H$ solves
\begin{align}
\pa_t\wh{\theta}_H-\mu(\pa_y^2-k^2)\wh{\theta}_H+ik\wh{\theta}_H=0, \ \wh{\theta}_H|_{t=0}=\wh{\theta}_{in},\ \wh{\theta}_H|_{y=\pm1}=0.
\end{align}

For the homogeneous part $\wh{\theta}_H$, by using transport diffusion structure and the Gearhart-Pr\"uss type lemma with sharp bound \cite{Wei},  we use the following estimates.
\begin{lemma}\label{lem: homo part}(Lemma 6.3 in \cite{CLWZ2020}.)
Let $\wh{\theta}_{in}\in L^2(-1,1)$. Then for any $k\in\bbZ$, there exist constants $C, c>0$ independent of $\mu,k$ such that 
\begin{align*}
\|\wh{\theta}_H\|_{L^2}\leq Ce^{-c\mu^{\f13}|k|^{\f23}t-\mu t}\|\wh{\theta}_{in}\|_{L^2}.
\end{align*}
Moreover, for any $|k|\geq1$,
\begin{align*}
(\mu k^2)^{\f13}\|\wh{\theta}_H\|^2_{L^2L^2}\leq C\|\wh{\theta}_{in}\|^2_{L^2}.
\end{align*}
\end{lemma}
For the inhomogeneous part, considering the system
 \begin{align}\label{eq: theta resolvent}
-\mu(\pa_y^2-k^2)\wh{\Theta}+ik(y-\la)\wh{\Theta}=F,\ \wh{\Theta}(\pm1)=0, 
\end{align}
we have the following sharp resolvent estimates for the linearized operator, which is very important for the space-time estimates of $\wh{\theta}_I$.
\begin{lemma}\label{lem: L2 H-1}\textsl{(Proposition 3.1 and Proposition 3.3 in \cite{CLWZ2020}.)}
Let $\wh{\Theta}\in H^2(-1,1)$ be a solution of \eqref{eq: theta resolvent} with $\la\in\bbR$. Then it holds for  $F\in L^2(-1,1)$,
\begin{align*}
\mu^{\f23}|k|^{\f13}\|\pa_y\wh{\Theta}\|_{L^2}+(\mu k^2)^{\f13}\|\wh{\Theta}\|_{L^2}+|k|\|(y-\la)\wh{\Theta}\|_{L^2}\leq C\|F\|_{L^2},
\end{align*}
and for $F\in H^{-1}(-1,1)$, 
 \begin{align*}
\mu\|\pa_y\wh{\Theta}\|_{L^2}+\mu^{\f23} |k|^{\f13}\|\wh{\Theta}\|_{L^2}\leq C\|F\|_{H^{-1}}.
\end{align*}
\end{lemma}
\begin{proposition}\label{prop: theta I}
Let $\wh{\theta}_I$ be a solution of \eqref{eq: theta I}. Then there exists a constant $C>0$ independent of $\mu,k$ such that 
\begin{align*}
(\mu k^2)^{\f13}\|\wh{\theta}_I\|_{L^2L^2}^2+\|\wh{\theta}_I\|_{L^{\infty}L^2}^2\leq C\big(\mu^{-\f13}|k|^{\f43}\|g^1\|_{L^2L^2}^2+\mu^{-1}\|g^2\|_{L^2L^2}^2\big).
\end{align*}
\end{proposition}
\begin{proof}
Now we use the resolvent estimates in Lemma \ref{lem: L2 H-1} to obtain the semigroup estimates.
By taking the Fourier transform in $t$:
\begin{align*}
\wh{\theta}(\la,k,y)
&=\int_0^{+\infty}\wh{\theta}_I(t,k,y)e^{-it\la}dt,\\
G^j(\la,k,y)&=\int_0^{+\infty}g^j(t,k,y)e^{-it\la}dt, \ j=1,2,
\end{align*}
we get that from \eqref{eq: theta I},
\begin{align}
(i\la-\mu(\pa_y^2-k^2)+iky)\wh{\theta}(\la,k,y)=-ikG^1(\la,k,y)-\pa_yG^2(\la,k,y).
\end{align}
Using Plancherel's theorem, we know that 
\begin{align*}
&\int_0^{+\infty}\|\wh{\theta}_I(t)\|_{L^2}^2dt\sim\int_{\bbR}\|\wh{\theta}(\la)\|_{L^2}^2d\la,\\
&\int_0^{+\infty}\|g^j(t)\|_{L^2}^2dt\sim\int_{\bbR}\|G^j(\la)\|_{L^2}^2d\la,\ j=1,2.
\end{align*}
We further decompose $\wh{\theta}_I=\wh{\theta}_I^{(1)}+\wh{\theta}_I^{(2)}$, where $\wh{\theta}_I^{(1)}$ and $\wh{\theta}_I^{(2)}$ solve
\begin{align*}
&(i\la-\mu(\pa_y^2-k^2)+iky)\wh{\theta}_I^{(1)}(\la,k,y)=-ikG^1(\la,k,y), \wh{\theta}_I^{(1)}|_{y=\pm1}=0,\\
&(i\la-\mu(\pa_y^2-k^2)+iky)\wh{\theta}_I^{(2)}(\la,k,y)=-\pa_yG^2(\la,k,y), \wh{\theta}_I^{(2)}|_{y=\pm1}=0.
\end{align*}
By Lemma \ref{lem: L2 H-1}, we get
\begin{align*}
(\mu k^2)^{\f13}\|\wh{\theta}_I^{(1)}(\la)\|_{L^2}\leq C\|kG^1(\la)\|_{L^2},
\end{align*}
and 
\begin{align*}
\mu^{\f23}|k|^{\f13}\|\wh{\theta}_I^{(2)}(\la)\|_{L^2}\leq C\|G^2(\la)\|_{L^2}.
\end{align*}
Then, by Plancherel's theorem, we have
\begin{align*}
&(\mu k^2)^{\f13}\|\wh{\theta}_I\|_{L^2L^2}^2\sim (\mu k^2)^{\f13}\Big\|\|\wh{\theta}_I(\la)\|_{L^2}\Big\|_{L^2(\bbR)}^2\\
&\quad \leq 2(\mu k^2)^{\f13}\Big(\Big\|\|\wh{\theta}_I^{(1)}(\la)\|_{L^2}\Big\|_{L^2(\bbR)}^2+\Big\|\|\wh{\theta}_I^{(2)}(\la)\|_{L^2}\Big\|_{L^2(\bbR)}^2\Big)\\
&\quad\leq C(\mu k^2)^{\f13}\Big(\Big\|(\mu k^2)^{-\f13}\|kG^1(\la)\|_{L^2}\Big\|_{L^2(\bbR)}^2+\Big\|\mu^{-\f23}|k|^{-\f13}\|G^2(\la)\|_{L^2}\Big\|_{L^2(\bbR)}^2\Big)\\
&\quad=C\mu^{-\f13}|k|^{\f43}\Big\|\|G^1(\la)\|_{L^2}\Big\|_{L^2(\bbR)}^2+C\mu^{-1}\Big\|\|G^2(\la)\|_{L^2}\Big\|_{L^2(\bbR)}^2\\
&\quad\sim \mu^{-\f13}|k|^{\f43}\|g^1\|_{L^2L^2}^2+\mu^{-1}\|g^2\|_{L^2L^2}^2.
\end{align*}
Next we estimate $\|\wh{\theta}_I\|_{L^{\infty}L^2}$. Notice that 
\begin{align*}
&\f12\pa_t\|\wh{\theta}_I\|_{L^2}^2+\mu\|\pa_y\wh{\theta}_I\|_{L^2}^2+\mu k^2\|\theta_I\|_{L^2}^2\\
&\quad=\text{Re}\langle(\pa_t-\mu(\pa_y^2-k^2)+iky)\wh{\theta}_I,\wh{\theta}_I\rangle\\
&\quad=\text{Re}\langle-ikg^1-\pa_yg^2,\wh{\theta}_I\rangle=\text{Re}\Big(-ik\langle g^1,\wh{\theta}_I\rangle+\langle g^2,\pa_y\wh{\theta}_I\rangle\Big)\\
&\quad\leq|k|\|g^1\|_{L^2}\|\wh{\theta}_I\|_{L^2}+\|g^2\|_{L^2}\|\pa_y\wh{\theta}_I\|_{L^2},
\end{align*}
which gives 
\begin{align*}
\pa_t\|\wh{\theta}_I\|_{L^2}^2+\mu\|\pa_y\wh{\theta}_I\|_{L^2}^2+2\mu k^2\|\wh{\theta}_I\|_{L^2}\leq \mu^{-\f13}|k|^{\f43}\|g^1\|_{L^2}^2+(\mu k^2)^{\f13}\|\wh{\theta}_I\|_{L^2}^2+\mu^{-1}\|g^2\|_{L^2}^2.
\end{align*}
As $\wh{\theta}_I|_{t=0}=0$, this shows that 
\begin{align*}
\|\wh{\theta}_I(t)\|_{L^2}^2&\leq\int_0^t\Big(\mu^{-\f13}|k|^{\f43}\|g^1(s)\|_{L^2}^2+(\mu k^2)^{\f13}\|\wh{\theta}_I(s)\|_{L^2}^2+\mu^{-1}\|g^2(s)\|_{L^2}^2\Big)ds\\
&\leq \mu^{-\f13}|k|^{\f43}\|g^1\|_{L^2L^2}^2+(\mu k^2)^{\f13}\|\wh{\theta}_I\|_{L^2L^2}^2+\mu^{-1}\|g^2\|_{L^2L^2}^2\\
&\leq C(\mu^{-\f13}|k|^{\f43}\|g^1\|_{L^2L^2}^2+\mu^{-1}\|g^2\|_{L^2L^2}^2).
\end{align*}
Thus, we get
\begin{align*}
\|\wh{\theta}_I\|_{L^{\infty}L^2}^2\leq C(\mu^{-\f13}|k|^{\f43}\|g^1\|_{L^2L^2}^2+\mu^{-1}\|g^2\|_{L^2L^2}^2).
\end{align*}

This completes the proof of Proposition \ref{prop: theta I}. 
\end{proof}

Thus, combining Lemma \ref{lem: homo part} and Proposition \ref{prop: theta I},  we immediately obtain the following space-time estimates of $\wh{\theta}$.
\begin{proposition}\label{prop: theta}
Let $\wh{\theta}$ be a solution of \eqref{eq: theta} with $\wh{\theta}_{in}\in L^2(-1,1)$ and $g^1,g^2\in L^2L^2$. Then there exists a constant $C>0$ independent in $\mu,k$ such that 
\begin{align*}
\|\wh{\theta}\|_{L^{\infty}L^2}^2+(\mu k^2)^{\f13}\|\wh{\theta}\|_{L^2L^2}^2\leq \|\wh{\theta}_{in}\|_{L^2}^2+C\Big(\mu^{-\f13}|k|^{\f43}\|g^1\|_{L^2L^2}^2+\mu^{-1}\|g^2\|_{L^2L^2}^2\Big).
\end{align*}
\end{proposition}
\iffalse
And from Proposition \ref{prop: theta energy esti} and Proposition \ref{prop: theta},  we also have the following corollary.
\begin{corol}\label{corol: theta}
Let $\wh{\theta}$ be a solution of \eqref{eq: theta} with $|k|^{\f13}\wh{\theta}_{in}\in L^2(-1,1)$ and $g^1,g^2\in L^2L^2$. Then there exists a constant $C>0$ independent in $\nu,k$ such that 
\begin{align*}
\|{\color{red}|k|^{\f13}}\wh{\theta}\|^2_{L^{\infty}L^2}+\nu k^2\||k|^{\f13}\wh{\theta}\|^2_{L^2L^2}\leq \||k|^{\f13}\wh{\theta}_{in}\|^2_{L^2}+C\nu^{-1}|k|^{\f23}\big(\|g^1\|^2_{L^2L^2}+\|g^2\|^2_{L^2L^2}\big),
\end{align*}
and
\begin{align*}
\||k|^{\f13}\wh{\theta}\|_{L^{\infty}L^2}^2+(\nu k^2)^{\f13}\||k|^{\f13}\wh{\theta}\|_{L^2L^2}^2\leq \||k|^{\f13}\wh{\theta}_{in}\|_{L^2}^2+C\Big(\nu^{-\f13}|k|^2\|g^1\|_{L^2L^2}^2+\nu^{-1}|k|^{\f23}\|g^2\|_{L^2L^2}^2\Big).
\end{align*}
\end{corol}
\fi
%%%%%%%%%%%%%%%%%%%%%%%%%%%%%%%%%%%

\section{Nonlinear stability}\label{sec: nonlinear}
In this section, we prove Theorem \ref{main thm}.  Due to the buoyancy term $\pa_x\theta$ in the equation of the vorticity, we need to estimate $\||D_x|^{\f16}\theta(t)\|_{L^2}$ in order to control the buoyancy term. In fact, for the two-dimensional Boussinesq equation, the global existence of smooth solution is well-known for the data $u_{in}\in H^2(\Om), \theta_{in}\in H^1(\Om)$ and $|D_x|^{\f16}\theta_{in}\in H^1(\Om)$. The main interest of Theorem \ref{main thm} is the stability estimates 
\begin{align}\label{main esti}
\sum_{k\in\bbZ}E_k\leq C\e_0\min\{\nu,\mu\}^{\f12}, \ \sum_{k\in\bbZ}H_k\leq C\e_1\min\{\nu,\mu\}^{\f{11}{12}}.
\end{align}
Here $E_0=\|\wh{w}_0\|_{L^{\infty}L^2}$ and $H_0=\|\wh{\theta}_0\|_{L^{\infty}L^2}$, and for $k\neq0$,
\begin{align*}
E_k=\|(1-|y|)^{\f12}\wh{w}_k\|_{L^{\infty}L^2}+|k|\|\wh{u}_k\|_{L^2L^2}+|k|^{\f12}\|\wh{u}_k\|_{L^{\infty}L^{\infty}}+(\nu k^2)^{\f14}\|\wh{w}_k\|_{L^2L^2},
\end{align*}
and 
\begin{align*}
H_k=|k|^{\f16}\|\wh{\theta}_k\|_{L^{\infty}L^2}+\mu^{\f16} |k|^{\f12}\|\wh{\theta}_k\|_{L^2L^2}.
\end{align*}
And we can get the following estimates, which along with bootstrap arguments, then we can easily deduce the estimates \eqref{main esti}.\begin{proposition}\label{prop: EFG} 
There hold that, for $k\neq0$, 
\begin{align}\label{Ek}
E_k\leq\|\wh{w}_{in,k}\|_{L^2}+|k|^{-1}\|\pa_y\wh{w}_{in,k}\|_{L^2}+C\nu^{-\f12}\sum_{l\in\bbZ}E_lE_{k-l}+C\nu^{-\f14}\mu^{-\f16}H_k,
\end{align}
and
\begin{align}\label{E0}
E_0\leq \|\wh{w}_{in,0}\|_{L^2}+C\nu^{-\f12}\sum_{l\in\bbZ\setminus\{0\}}E_lE_{-l}.
\end{align}
For $H_0$, there holds that 
\begin{align}\label{F0}
H_0\lesssim\|\wh{\theta}_{in,0}\|_{L^2}+\mu^{-\f12}\sum_{l\in\bbZ\setminus\{0\}}|l|^{-\f23}E_lH_{-l}.
\end{align}
For $k\neq0$, there hold that \\
1. for $\mu k^2\leq1$, 
\begin{align}\label{esti: F low}
H_k\lesssim|k|^{\f16}\|\wh{\theta}_{in,k}\|_{L^2}+\mu^{-\f12}\sum_{l\in\bbZ}E_lH_{k-l}+\nu^{-\f18}\mu^{-\f{5}{24}}\sum_{l\in\bbZ\setminus\{0,k\},|k-l|\leq\f{|k}{2}}E_lH_{k-l};
\end{align}
2. for $\mu k^2>1$, 
\begin{align}\label{esti: F high}
H_k\lesssim |k|^{\f16}\|\wh{\theta}_{in,k}\|_{L^2}+\mu^{-\f12}\sum_{l\in\bbZ}E_lH_{k-l}+\nu^{-\f18}\mu^{-\f{5}{24}}E_kH_0.
\end{align}
 \end{proposition}
\begin{proof}
\no\textsl{Proof of \eqref{Ek}.}\
Denoting $\wh{w}_k(t,y)=\frac{1}{2\pi}\int_{\bbT}\om(t,x,y)e^{-ikx}dx$ and 
\begin{align*}
f_k^1(t,y)&=\sum_{l\in\bbZ}\wh{u}_l^1(t,y)\wh{w}_{k-l}(t,y),\ f_k^2(t,y)=\sum_{l\in\bbZ}\wh{u}_l^2(t,y)\wh{w}_{k-l}(t,y),
\end{align*}
we have 
\begin{align}\label{eq: w non}
(\pa_t-\nu(\pa_y^2-k^2)+iky)\wh{w}_k(t,y)&=-ik\wh{\theta}_k(t,y)-ikf_k^1(t,y)-\pa_yf_k^2(t,y).
\end{align}
It follows from Proposition \ref{prop: w} that
\begin{align}\label{E}
E_k&\leq C\Big(\nu^{-\f14}|k|^{\f12}\|\wh{\theta}_k\|_{L^2L^2}+\nu^{-\f14}|k|^{\f12}\|f_k^1\|_{L^2L^2}+\nu^{-\f12}\|f_k^2\|_{L^2L^2}\Big)\nonumber\\
&\quad+\|\wh{w}_{in,k}\|_{L^2}+|k|^{-1}\|\pa_y\wh{w}_{in,k}\|_{L^2}.
\end{align}

As in \cite{CLWZ2020}, we get that for $k\neq0$,   
\begin{align*}
\left\|\f{\wh{u}_k^2(t,y)}{(1-|y|)^{\f12}}\right\|^2_{L^2L^{\infty}}&=\left\|\sup_{y\in[-1,1]}\f{|\wh{u}_k^2(t,y)|^2}{1-|y|}\right\|_{L^1}\\
&=\left\|\max\{\sup_{y\in[0,1]}\f{|\int_1^y\pa_z\wh{u}_k^2(t,z)dz|^2}{1-|y|},\sup_{y\in[-1,0]}\f{|\int_{-1}^y\pa_z\wh{u}_k^2(t,z)dz|^2}{1-|y|}\}\right\|_{L^1}\\
&\leq 4\|\pa_y\wh{u}_k^2\|_{L^2L^2}^2=4|k|^2\|\wh{u}_k^1\|_{L^2L^2}^2\leq 4E_k^2.
\end{align*}
From which, we infer that, for $k\in\bbZ$,
\begin{align}\label{esti: f2}
\|f_k^2\|_{L^2L^2}\leq\sum_{l\in\bbZ}\left\|\f{\wh{u}_l^2(t,y)}{(1-|y|)^{\f12}}\right\|_{L^2L^{\infty}}\|(1-|y|)^{\f12}\wh{w}_{k-l}\|_{L^{\infty}L^2}\leq 2\sum_{l\in\bbZ}E_lE_{k-l},
\end{align}
and
\begin{align*}
\|f_k^1\|_{L^2L^2}\leq\|\wh{u}_0^1\|_{L^{\infty}L^{\infty}}\|\wh{w}_k\|_{L^2L^2}+\|\wh{u}_k^1\|_{L^2L^{\infty}}\|\wh{w}_0\|_{L^{\infty}L^2}+\sum_{l\in\bbZ\setminus\{0,k\}}\|\wh{u}_l^1\|_{L^{\infty}L^{\infty}}\|\wh{w}_{k-l}\|_{L^2L^2}.
\end{align*}
Thanks to $|l||k-l|\gtrsim|k|(l\neq0,k)$, we have
\begin{align*}
\sum_{l\in\bbZ\setminus\{0,k\}}\|\wh{u}_l^1\|_{L^{\infty}L^{\infty}}\|\wh{w}_{k-l}\|_{L^2L^2}&\lesssim \sum_{l\in\bbZ\setminus\{0,k\}}|l|^{-\f12}E_l\nu^{-\f14}|k-l|^{-\f12}E_{k-l}\\
&\lesssim |k|^{-\f12}\nu^{-\f14}\sum_{l\in\bbZ\setminus\{0,k\}}E_lE_{k-l},
\end{align*}
and
\begin{align*}
\|\wh{u}_0^1\|_{L^{\infty}L^{\infty}}\|\wh{w}_k\|_{L^2L^2}+\|\wh{u}_k^1\|_{L^2L^{\infty}}\|\wh{w}_0\|_{L^{\infty}L^2}&\lesssim\|\wh{w}_0\|_{L^{\infty}L^2}\|\wh{w}_k\|_{L^2L^2}\lesssim (\nu k^2)^{-\f14}E_kE_0.
\end{align*}
This shows that 
\begin{align}\label{f1}
\|f_k^1\|_{L^2L^2}\lesssim (\nu k^2)^{-\f14}\sum_{l\in\bbZ}E_lE_{k-l}.
\end{align}

Thus, by \eqref{E}, \eqref{esti: f2} and \eqref{f1}, we obtain that 
\begin{align*}
E_k\leq\|\wh{w}_{in,k}\|_{L^2}+|k|^{-1}\|\pa_y\wh{w}_{in,k}\|_{L^2}+C\nu^{-\f12}\sum_{l\in\bbZ}E_lE_{k-l}+C\nu^{-\f14}\mu^{-\f16}H_k.
\end{align*}
\no\textsl{Proof of \eqref{E0}.}\ Due to $\text{div}\,u=0$, we have $\wh{u}_0^2(t,y)=0$. By $P_0(\wh{u}^1\pa_x\wh{u}^1)=0$, we have
\begin{align}\label{eq: ol u}
\pa_t\wh{u}_0^1(t,y)-\nu\pa_y^2\wh{u}_0^1(t,y)&=-\sum_{l\in\bbZ\setminus\{0\}}\wh{u}_l^2(t,y)\pa_y\wh{u}_{-l}^1(t,y)\nonumber\\
&=-\sum_{l\in\bbZ\setminus\{0\}}\wh{u}_l^2(t,y)\wh{w}_{-l}(t,y)=-f_0^2(t,y).
\end{align}
 By integration by parts in \eqref{eq: ol u}, we get
\begin{align*}
\langle(\pa_t-\nu\pa_y^2)\wh{u}_0^1,-\pa_y^2\wh{u}_0^1\rangle=\f12\pa_t\|\pa_y\wh{u}_0^1(t)\|_{L^2}^2+\nu\|\pa_y^2\wh{u}_0^1(t)\|_{L^2}^2=\langle f_0^2,\pa_y^2\wh{u}_0^1\rangle,
\end{align*}
which gives 
\begin{align*}
\pa_t\|\pa_y\wh{u}_0^1(t)\|_{L^2}^2+\nu\|\pa_y^2\wh{u}_0^1(t)\|_{L^2}^2\leq C\nu^{-1}\|f_0^2(t,y)\|_{L^2}^2,
\end{align*}
from which, along with $\pa_y\wh{u}_0^1(t,y)=\wh{w}_0(t,y)$, we infer that 
\begin{align}\label{eq: E0}
E_0^2=\|\wh{w}_0\|^2_{L^{\infty}L^2}\leq C \nu^{-1}\|f_0^2(t,y)\|_{L^2L^2}^2+\|\wh{w}_{in,0}\|_{L^2}^2.
\end{align}

Thus, by using \eqref{esti: f2}, we obtain
\begin{align*}
E_0\leq \|\wh{w}_{in,0}\|_{L^2}+C\nu^{-\f12}\sum_{l\in\bbZ\setminus\{0\}}E_lE_{-l}.
\end{align*}
\no\textsl{Proof of \eqref{F0}.}\  Similarly, we can derive the evolution equation of $\wh{\theta}_0$,
\begin{align}\label{eq: ol theta}
\pa_t\wh{\theta}_0-\mu\pa_y^2\wh{\theta}_0=-\sum_{l\in\bbZ\setminus\{0\}}\pa_y(u_l^2\wh{\theta}_{-l})(t,y)=-\pa_yg_0^2(t,y).
\end{align}

Similarly as the estimate of $E_0$, we get that by integration by parts in \eqref{eq: ol theta},
\begin{align}\label{eq:F0}
H_0^2=\|\wh{\theta}_0\|^2_{L^{\infty}L^2}+\mu\|\pa_y\wh{\theta}_0\|^2_{L^2L^2}\leq {C\mu^{-1}\|g_0^2(t,y)\|_{L^2L^2}^2}+\|\wh{\theta}_{in,0}\|_{L^2}^2.
\end{align}
By using the Gagliardo-Nirenberg inequality and $\pa_y\wh{u}_k^2=-ik\wh{u}_k^1$, we have
\begin{align}\label{G-N u2}
\|\wh{u}_k^2\|_{L^2L^{\infty}}\leq C|k|^{\f12}\|\wh{u}_k^2\|_{L^2L^2}^{\f12}\|\wh{u}_k^1\|_{L^2L^2}^{\f12}\leq C |k|^{-\f12}E_k.
\end{align}
And then, we obtain
\begin{align}\label{esti: g0}
\|g_0^2\|_{L^2L^2}&\leq\sum_{l\in\bbZ\setminus\{0\}}\|\wh{u}_l^2\|_{L^2L^{\infty}}\|\wh{\theta}_{-l}\|_{L^{\infty}L^2}\lesssim\sum_{l\in\bbZ\setminus\{0\}}|l|^{-\f12}|-l|^{-\f16}E_lH_{-l}\nonumber\\
&\lesssim\sum_{l\in\bbZ\setminus\{0\}}|l|^{-\f23}E_lH_{-l}.
\end{align}
Thus, from  \eqref{eq:F0} and \eqref{esti: g0},  we have
\begin{align*}
H_0\lesssim\|\wh{\theta}_{in,0}\|_{L^2}+\mu^{-\f12}\sum_{l\in\bbZ\setminus\{0\}}|l|^{-\f23}E_lH_{-l}.
\end{align*}

In order to control the nonlinear term $\|g_k^1\|_{L^2L^2}$, during the estimates of $H_k$, we need to divide them into the low frequency part $\mu k^2\leq 1$ and the high frequency part $\mu k^2>1$.

%%%%%%%%%%%%%%%%%%%%%%%%%%低频部分的处理
\no\textsl{Proof of \eqref{esti: F low}.}\ First, we can derive the evolution equations of $\wh{\theta}_k(t,y)=\frac{1}{2\pi}\int_{\bbT}\theta(t,x,y)e^{-ikx}dx$. Denoting 
\begin{align*}
g_k^1(t,y)&=\sum_{l\in\bbZ}\wh{u}_l^1(t,y)\wh{\theta}_{k-l}(t,y),\ \ g_k^2(t,y)=\sum_{l\in\bbZ}\wh{u}_l^2(t,y)\wh{\theta}_{k-l}(t,y),
\end{align*}
we have that $\wh{\theta}_k(t,y)$ satisfies,
\begin{align}
(\pa_t-\mu(\pa_y^2-k^2)+iky)\wh{\theta}_k(t,y)=-ikg_k^1(t,y)-\pa_yg_k^2(t,y).
\end{align}
For $\mu k^2\leq 1$, it follows from Proposition \ref{prop: theta} that
\begin{align}\label{F}
H_k \leq|k|^{\f16} \|\wh{\theta}_{in,k}\|_{L^2}+C\Big(\mu^{-\f16}|k|^{\f56}\|g_k^1\|_{L^2L^2}+\mu^{-\f12}|k|^{\f16}\|g_k^2\|_{L^2L^2}\Big).
\end{align}

On the one hand, by using $\wh{u}_0^2=0$ and \eqref{G-N u2}, we have that for $k\neq0$,
\begin{align}\label{esti: F g2}
\|g_k^2\|_{L^2L^2}&\leq\|\wh{u}_k^2\|_{L^2L^{\infty}}\|\wh{\theta}_0\|_{L^{\infty}L^2}+\sum_{l\in\bbZ\setminus\{0,k\}}\|\wh{u}_l^2\|_{L^2L^{\infty}}\|\wh{\theta}_{k-l}\|_{L^{\infty}L^2}\nonumber\\
&\leq |k|^{-\f12}E_kH_0+\sum_{l\in\bbZ\setminus\{0,k\}}|l|^{-\f12}|k-l|^{-\f16}E_lH_{k-l}\nonumber\\
&\leq |k|^{-\f12}E_kH_0+|k|^{-\f16}\sum_{l\in\bbZ\setminus\{0,k\}}E_lH_{k-l}.\end{align}

On the other hand, for $g_k^1$ and $k\neq0$, by using Gagliardo-Nirenberg inequality, we have
\begin{align}\label{11}
\|\wh{u}_k^1\|_{L^2L^{\infty}}\leq C\|\wh{u}_k^1\|_{L^2L^2}^{\f12}\|\pa_y\wh{u}_k^1\|_{L^2L^2}^{\f12}\leq C\nu^{-\f18}|k|^{-\f34}E_k,
\end{align}
 and then we obtain that 
\begin{align}\label{22}
\|g_k^1\|_{L^2L^2}&\leq\|\wh{u}_0^1\|_{L^{\infty}L^{\infty}}\|\wh{\theta}_k\|_{L^2L^2}+\|\wh{u}_k^1\|_{L^2L^{\infty}}\|\wh{\theta}_0\|_{L^{\infty}L^2}+\sum_{l\in\bbZ\setminus\{0,k\}}\|\wh{u}_l^1\wh{\theta}_{k-l}\|_{L^2L^2}\nonumber\\
&\leq\|\wh{w}_0\|_{L^{\infty}L^2}\|\wh{\theta}_k\|_{L^2L^2}+\nu^{-\f18}|k|^{-\f34}E_k\|\wh{\theta}_0\|_{L^{\infty}L^2}+\sum_{l\in\bbZ\setminus\{0,k\}}\|\wh{u}_l^1\wh{\theta}_{k-l}\|_{L^2L^2}\nonumber\\
&\leq \mu^{-\f16}|k|^{-\f12}E_0H_k+\nu^{-\f18}|k|^{-\f34}E_kH_0+\sum_{l\in\bbZ\setminus\{0,k\}}\|\wh{u}_l^1\wh{\theta}_{k-l}\|_{L^2L^2}.
\end{align} 
To estimate $\sum\limits_{l\in\bbZ\setminus\{0,k\}}\|\wh{u}_l^1\wh{\theta}_{k-l}\|_{L^2L^2}$, we divide it into two parts and get that
\begin{align*}
\sum_{l\in\bbZ\setminus\{0,k\}}\|\wh{u}_l^1\wh{\theta}_{k-l}\|_{L^2L^2}&\leq\sum_{l\in\bbZ\setminus\{0,k\},|k-l|\leq\f{|k|}{2}}\|\wh{u}_l^1\wh{\theta}_{k-l}\|_{L^2L^2}+\sum_{l\in\bbZ\setminus\{0,k\},|k-l|>\f{|k|}{2}}\|u_l^1\wh{\theta}_{k-l}\|_{L^2L^2}\nonumber\\
&\eqdef\text{HL}+\text{LH}, 
\end{align*}
whereas by \eqref{11}, 
\begin{align}\label{HL}
\text{HL}&\leq\sum_{l\in\bbZ\setminus\{0,k\},|k-l|\leq\f{|k|}{2}}\|\wh{u}_l^1\|_{L^2L^{\infty}}\|\wh{\theta}_{k-l}\|_{L^{\infty}L^2}\nonumber\\
&\lesssim \sum_{l\in\bbZ\setminus\{0,k\},|k-l|\leq\f{|k|}{2}}\nu^{-\f18}|l|^{-\f34}|k-l|^{-\f16}E_lH_{k-l}\nonumber\\
&\lesssim \nu^{-\f18}|k|^{-\f34}\sum_{l\in\bbZ\setminus\{0,k\},|k-l|\leq\f{|k|}{2}}E_lH_{k-l},
\end{align}
and 
\begin{align}\label{LH}
\text{LH}&\leq\sum_{l\in\bbZ\setminus\{0,k\},|k-l|>\f{|k|}{2}}\|\wh{u}_l^1\|_{L^{\infty}L^{\infty}}\|\wh{\theta}_{k-l}\|_{L^2L^2}\nonumber\\
&\lesssim\sum_{l\in\bbZ\setminus\{0,k\},|k-l|>\f{|k|}{2}}|l|^{-\f12}|k-l|^{-\f12}\mu^{-\f16}E_lH_{k-l}\nonumber\\
&\lesssim\mu^{-\f16}|k|^{-\f12}\sum_{l\in\bbZ\setminus\{0,k\},|k-l|>\f{|k|}{2}}E_lH_{k-l}.
\end{align}
And then, substituting \eqref{HL} and \eqref{LH} into \eqref{22}, we get
\begin{align}\label{esti: F g1}
\|g_k^1\|_{L^2L^2}&\lesssim \mu^{-\f16}|k|^{-\f12}E_0H_k+\nu^{-\f18}|k|^{-\f34}E_kH_0\nonumber\\
&\quad+\nu^{-\f18}|k|^{-\f34}\sum_{l\in\bbZ\setminus\{0,k\},|k-l|\leq\f{|k|}{2}}E_lH_{k-l}+\mu^{-\f16}|k|^{-\f12}\sum_{l\in\bbZ\setminus\{0,k\},|k-l|>\f{|k|}{2}}E_lH_{k-l}.
\end{align}
Thus, combining \eqref{F}, \eqref{esti: F g2} and \eqref{esti: F g1}, we get that for $k\neq0$ and $\mu k^2 \leq1$,
\begin{align*}%\label{esti: F low}
H_k\lesssim& |k|^{\f16}\|\wh{\theta}_{in,k}\|_{L^2}+\mu^{-\f12}\sum_{l\in\bbZ\setminus\{0\}}E_lH_{k-l}+\mu^{-\f13}|k|^{\f13}E_0H_k+\nu^{-\f18}\mu^{-\f16}|k|^{\f{1}{12}}E_kH_0\nonumber\\
&\quad+\nu^{-\f18}\mu^{-\f16}|k|^{\f{1}{12}}\sum_{l\in\bbZ\setminus\{0,k\},|k-l|\leq\f{|k|}{2}}E_lH_{k-l}+\mu^{-\f13}|k|^{\f13}\sum_{l\in\bbZ\setminus\{0,k\},|k-l|>\f{|k|}{2}}E_lH_{k-l}\nonumber\\
&\lesssim|k|^{\f16} \|\wh{\theta}_{in,k}\|_{L^2}+\mu^{-\f12}\sum_{l\in\bbZ}E_lH_{k-l}+\nu^{-\f18}\mu^{-\f{5}{24}}E_kH_0\nonumber\\
&\quad +\nu^{-\f18}\mu^{-\f{5}{24}}\sum_{l\in\bbZ\setminus\{0,k\},|k-l|\leq\f{|k|}{2}}E_lH_{k-l}+\mu^{-\f12}\sum_{l\in\bbZ\setminus\{0,k\},|k-l|>\f{|k|}{2}}E_lH_{k-l}\nonumber\\
&\lesssim|k|^{\f16}\|\wh{\theta}_{in,k}\|_{L^2}+\mu^{-\f12}\sum_{l\in\bbZ}E_lH_{k-l}+\nu^{-\f18}\mu^{-\f{5}{24}}\sum_{l\in\bbZ\setminus\{0,k\},|k-l|\leq\f{|k|}{2}}E_lH_{k-l}.
\end{align*}
%%%%%%%%%%%%%%%%%%%%%%%%%高频部分处理
\no\textsl{Proof of \eqref{esti: F high}.} For  $\mu k^2>1$, it follows from Proposition \ref{prop: theta energy esti} that
\begin{align}\label{F high fre}
H_k &\leq |k|^{\f16}\|\wh{\theta}_k\|_{L^{\infty}L^2}+|k|^{\f16}(\mu k^2)^{\f12}\|\wh{\theta}_k\|_{L^2L^2}\nonumber\\
&\leq |k|^{\f16}\|\wh{\theta}_{in,k}\|_{L^2}+C\mu^{-\f12}|k|^{\f16}\big(\|g_k^1\|_{L^2L^2}+\|g_k^2\|_{L^2L^2}\big).
%&\|\theta_{0,k}\|_{L^2}+C\Big(\nu^{-\f16}|k|^{\f23}\|g_k^1\|_{L^2L^2}+\nu^{-\f12}\|g_k^2\|_{L^2L^2}\Big),
\end{align}

For $g_k^1$ and $k\neq0$,  by \eqref{11}, we obtain
\begin{align}\label{44}
\|g_k^1\|_{L^2L^2}&\leq\|\wh{u}_0^1\|_{L^{\infty}L^{\infty}}\|\wh{\theta}_k\|_{L^2L^2}+\|\wh{u}_k^1\|_{L^2L^{\infty}}\|\wh{\theta}_0\|_{L^{\infty}L^2}+\sum_{l\in\bbZ\setminus\{0,k\}}\|\wh{u}_l^1\wh{\theta}_{k-l}\|_{L^2L^2}\nonumber\\
&\lesssim \|\wh{w}_0\|_{L^{\infty}L^2}\|\wh{\theta}_k\|_{L^2L^2}+\nu^{-\f18}|k|^{-\f34}E_k\|\wh{\theta}_0\|_{L^{\infty}L^2}+\sum_{l\in\bbZ\setminus\{0,k\}}\|\wh{u}_l^1\wh{\theta}_{k-l}\|_{L^2L^2}\nonumber\\
&\lesssim\mu^{-\f16}|k|^{-\f12}E_0H_k+\nu^{-\f18}|k|^{-\f34}E_kH_0+\sum_{l\in\bbZ\setminus\{0,k\}}\|\wh{u}_l^1\wh{\theta}_{k-l}\|_{L^2L^2}.
\end{align} 
Whereas for the term $\sum_{l\in\bbZ\setminus\{0,k\}}\|\wh{u}_l^1\wh{\theta}_{k-l}\|_{L^2L^2}$, we can obtain that by using $|l||k-l|\gtrsim|k| (l\neq 0,k)$,
\begin{align*}
\sum_{l\in\bbZ\setminus\{0,k\}}\|\wh{u}_l^1\wh{\theta}_{k-l}\|_{L^2L^2}&\lesssim \sum_{l\in\bbZ\setminus\{0,k\}}\|\wh{u}_l^1\|_{L^{\infty}L^{\infty}}\|\wh{\theta}_{k-l}\|_{L^2L^2}\\
&\lesssim \sum_{l\in\bbZ\setminus\{0,k\}}|l|^{-\f12}E_l \mu^{-\f16}|k-l|^{-\f12}H_{k-l}\\
&\lesssim\mu^{-\f16}|k|^{-\f12}\sum_{l\in\bbZ\setminus\{0,k\}}E_lH_{k-l}.
\end{align*}
And then, we obtain
\begin{align}\label{esti: F g1 high}
\|g_k^1\|_{L^2L^2}&\lesssim \mu^{-\f16}|k|^{-\f12}E_0H_k+\nu^{-\f18}|k|^{-\f34}E_kH_0+\mu^{-\f16}|k|^{-\f12}\sum_{l\in\bbZ\setminus\{0,k\}}E_lH_{k-l}.
\end{align}

Thus, combining \eqref{F high fre}, \eqref{esti: F g2} and \eqref{esti: F g1 high}, we get that for $k\neq0$ and $\mu k^2>1$,
\begin{align*}
H_k&\lesssim |k|^{\f16}\|\wh{\theta}_{in,k}\|_{L^2}+\mu^{-\f12}\sum_{l\in\bbZ}E_lH_{k-l}+\mu^{-\f23}|k|^{-\f13}E_0H_k+\nu^{-\f18}\mu^{-\f12}|k|^{-\f{7}{12}}E_kH_0\nonumber\\
&\quad+\mu^{-\f23}|k|^{-\f13}\sum_{l\in\bbZ\setminus\{0,k\}}E_lH_{k-l}\nonumber\\
&\lesssim |k|^{\f16}\|\wh{\theta}_{in,k}\|_{L^2}+\mu^{-\f12}\sum_{l\in\bbZ}E_lH_{k-l}+\mu^{-\f12}E_0H_k+\nu^{-\f18}\mu^{-\f{5}{24}}E_kH_0+\mu^{-\f12}\sum_{l\in\bbZ\setminus\{0,k\}}E_lH_{k-l}\nonumber\\
&\lesssim |k|^{\f16}\|\wh{\theta}_{in,k}\|_{L^2}+\mu^{-\f12}\sum_{l\in\bbZ}E_lH_{k-l}+\nu^{-\f18}\mu^{-\f{5}{24}}E_kH_0.
\end{align*}
%%%%%%%%%%%%%%%%%%

This completes the proof of Proposition \ref{prop: EFG}.
\end{proof}

Now we prove Theorem \ref{main thm}. From \eqref{E0} and \eqref{Ek}, we deduce
\begin{align}\label{esti: E sum}
\sum_{k\in\bbZ}E_k&\leq\sum_{k\in\bbZ}\|\wh{w}_{in,k}\|_{L^2}+\sum_{k\in\bbZ\setminus\{0\}}|k|^{-1}\|\pa_y\wh{w}_{in,k}\|_{L^2}\nonumber\\
&\quad +C\nu^{-\f12}\sum_{k\in\bbZ}\sum_{l\in\bbZ}E_lE_{k-l}+C\nu^{-\f14}\mu^{-\f16}\sum_{k\in\bbZ\setminus\{0\}}H_k.
\end{align}
And by the fact that 
$$\sum_{k\in\bbZ}H_k=H_0+\sum_{k\in\bbZ\setminus\{0\}, \mu k^2\leq1}H_k+\sum_{k\in\bbZ\setminus\{0\}, \mu k^2>1}H_k,
$$  
combining \eqref{F0}, \eqref{esti: F low} and \eqref{esti: F high}, we can deduce that 
\begin{align}\label{esti: F sum}
\sum_{k\in\bbZ}H_k&\lesssim\|\wh{\theta}_{in,0}\|_{L^2}+\sum_{k\in\bbZ\setminus\{0\}}|k|^{\f16}\|\wh{\theta}_{in,k}\|_{L^2}+\mu^{-\f12}\sum_{k\in\bbZ}\sum_{l\in\bbZ}E_lH_{k-l}\nonumber\\
&\quad+\nu^{-\f18}\mu^{-\f{5}{24}}\sum_{k\in\bbZ\setminus\{0\}, \mu k^2\leq1}\sum_{l\in\bbZ}E_kH_{k-l}+\nu^{-\f18}\mu^{-\f{5}{24}}\sum_{k\in\bbZ\setminus\{0\}, \mu k^2>1}E_kH_0.
\end{align}

 On the other hand, it is easy to verify that from $\|u_{in}\|_{H^2}\leq\e_0\min\{\nu,\mu\}^{\f12}$ and $\|\theta_{in}\|_{H^1}+\||D_x|^{\f16}\theta_{in}\|_{H^1}\leq \e_1\min\{\nu,\mu\}^{\f{11}{12}}$,
\begin{align*}
\sum_{k\in\bbZ}\|\wh{w}_{in,k}\|_{L^2}+\sum_{k\in\bbZ\setminus\{0\}}|k|^{-1}\|\pa_y\wh{w}_{in,k}\|_{L^2}\leq C\e_0\min\{\nu,\mu\}^{\f12},
\end{align*}
and 
\begin{align*}
 \|\wh{\theta}_{in,0}\|_{L^2}+\sum_{k\in\bbZ\setminus\{0\}}|k|^{\f16}\|\wh{\theta}_{in,k}\|_{L^2}\leq C\e_1\min\{\nu,\mu\}^{\f{11}{12}}.
\end{align*}

Thus, for  $\e_0,\e_1$ suitably small,  by bootstrap arguments, we can deduce from \eqref{esti: E sum} and \eqref{esti: F sum} that
\begin{align*}
\sum_{k\in\bbZ}E_k\leq C\e_0\min\{\nu,\mu\}^{\f12},\quad  \sum_{k\in\bbZ}H_k\leq C\e_1\min\{\nu,\mu\}^{\f{11}{12}}.
\end{align*}

This completes the proof of Theorem \ref{main thm}. \qed

\section*{Acknowledgements}
The work of N. Masmoudi is supported by NSF grant DMS-1716466 and by Tamkeen under the NYU Abu Dhabi Research Institute grant of the center SITE. C. Zhai's work is supported by a grant from the China Scholarship Council and this work was done when C. Zhai was visiting the center SITE, NYU Abu Dhabi. She appreciates the hospitality from NYU.

\end{CJK*}
\end{document}